\documentclass[12pt]{article}
\usepackage[left=1in,right=1in, top=1.2in,bottom=1.2in]{geometry}
\usepackage{here}
\usepackage{times}
\usepackage{comment}
\usepackage{amssymb,amsthm,latexsym,amsmath,epsfig,pgf}

\newtheorem{theorem}{Theorem}[section]
\newtheorem{lemma}{Lemma}[section]

\newtheorem{example}{Example}

\theoremstyle{remark}
\newtheorem{definition}{Definition}[section]
\newtheorem{remark}{Remark}[section]
\theoremstyle{remark}

\begin{document}

\begin{center}
\Large{ Games of Nim with Dynamic Restrictions }\\
\vspace{0.5cm}
\large{Keita Mizugaki,Shoei Takahashi,Hikaru Manabe}\\
\large{Aoi Murakami,Ryohei Miyadera}
\end{center}

\begin{abstract}
 The authors present formulas for the previous player's winning positions of two variants of restricted Nim.
 In both of these two games, there is one pile of stones, and 
 in the first variant, we investigate the case that in $k$-th turn, you can remove $f(k)$ stones at most, where  $f$ is a function whose values are natural numbers. In the second variant, there are two kinds of stones. The Type 1 group consists of stones with the weight of one, and the Type 2 group consists of stones with the weight of two. When the total weight of stones is $a$, you can remove stones whose total weight is equal to or less than $\lfloor \frac{a}{2} \rfloor$. 

\end{abstract}

\section{Introduction and a Restricted Nim}
 The classic game of Nim is played with stone piles. A player can remove any number of stones from any one pile during their turn; the player who takes the last stone is considered the winner. 
 
 There are many variants of the classical game of Nim. In this study, we investigate the game of Nim with one pile. When there is only one pile, we must restrict the number of stones players can remove in each turn. 

 In Maximum Nim which is one of restricted games of Nim, we place an upper bound $f(n)$ on the number of stones that players can remove in terms of the number $n$ of stones in the pile. As an example of restricted Nim, See Levine \cite{levinenim}, and Miyadera et al. \cite{thaij2023b} is a recent result.
 There is also a study on the restricted Nim with three piles. See Miyadera and Manabe 
\cite{integer2023}.

In this study, we investigate two variants of restricted Nim. 
In the first game, players can remove 
$f(k)$ stones at most in their $k$-th turn of the game, where $f$ is a function whose values are natural numbers. The authors present formulas for the previous player's winning positions of this game. This restriction is dynamic since the restriction condition will change as the game proceeds.

In the second, there are stones with a weight 1 and stones with a weight of 2. 
Players can remove stones whose total weight is equal to or less than half of the total weight in the pile. The restriction of this game depends not only on the total weight of the stones, but also on the numbers of Type 1 and Type 2 stones. We may call this restriction dynamic since it depends on many factors.
The authors present formulas for the previous player's winning positions of this game, too.

The significance of the present article is the introduction of two new games. Although the research results of these games are still in 
the beginning stage, the authors expect a bright future for the research of these games.

Let $\mathbb{Z}_{\ge 0}$ and  $\mathbb{N}$ represent the sets of non-negative integers and natural numbers, respectively.

The restricted nims that we study in the present article are impartial games without draws, there will be only
two kinds of positions.
\begin{definition}\label{NPpositions}
$(a)$ A position is referred to as a $\mathcal{P}$-\textit{position} if it is a winning position for the previous player (the player who just moved), as long as he/she plays correctly at every stage.\\
$(b)$ A position is referred to as an $\mathcal{N}$-\textit{position} if it is a winning position for the next player, as long as he/she plays correctly at every stage.
\end{definition}

\section{When Restriction Depends on the Turns}
In this section, we study a restricted game of Nim in which players can remove 
$f(k)$ stones at most in their $k$-th turn of the game, where $f$ is a function whose values are natural numbers.
\begin{definition}\label{gameofmizugakig}
Let $f$ be a monotonically increasing function of $\mathbb{N}$ into $\mathbb{N}$.
Suppose there is a pile of stones, and two players take turns removing stones from the pile.
In $k$th turn, the player is allowed to remove at least one stone and at most $f(k)$ stones. The player who removes the last stone is the winner. 
\end{definition}

\begin{definition}\label{gameofmizugaki1}
We denote by $(n,k)$ the position of the game when a player removes a stones in the $k$-th turn and the number of stones is $n$,
\end{definition}

\begin{definition}\label{moveofmizugaki}
For $u, k \in \mathbb{N}$, the set of all the positions that can be reached from position $(u,k)$ is defined as 
$\textit{move}(u,k)$.\\
$(i)$ We define 
\begin{equation}
\textit{move}(0,k)	= \{\} = \emptyset.\nonumber   
\end{equation}
$(ii)$
For $u,k \in \mathbb{N}$, 
\begin{equation}
\textit{move}(u,k)	= \{(u-t,k+1):t \in \mathbb{N} \text{ and } 1 \leq t \leq \min(u,f(k)) \}.\nonumber   
\end{equation}
\end{definition}

\begin{definition}\label{defofpkng}
For $n \in \mathbb{Z}_{\ge 0}$ and $k \in \mathbb{N}$, let 
\begin{equation}
\mathcal{P}_{k,n}^f = \{ (x,k):\sum_{t=1}^n (f(k+2t-2)+1) \leq x \leq  \sum_{t=1}^n (f(k+2t-1)+1) \}.
\end{equation}

\begin{equation}
\mathcal{P}_{k}^f = \cup \{\mathcal{P}_{k,n}^f:n \in \mathbb{Z}_{\ge 0} \}, \nonumber
\end{equation}
and 
\begin{equation}
\mathcal{P}^f= \cup \{\mathcal{P}_k^f:k \in \mathbb{N}\}.\nonumber
\end{equation}
\end{definition}

\begin{remark}
Note that 
$\mathcal{P}_{k,0}^f =\{(0,k): k \in \mathbb{Z}_{\ge 0}\}.$
\end{remark}

\begin{lemma}\label{comparenn1}
For $(u,k) \in \mathcal{P}_{k,n}^f$ and $(v,k) \in \mathcal{P}_{k,n+1}^f$, 
we have 
\begin{equation}
u < v.\label{usomallthvg}    
\end{equation}
\end{lemma}
\begin{proof}
Since $f$ is a monotonically increasing function of $\mathbb{N}$ into $\mathbb{N}$, by Definition \ref{defofpkng}, 
\begin{equation}
u \leq \sum_{t=1}^n (f(k+2t-1)+1) < \sum_{t=1}^{n+1} (f(k+2t-2)+1) \leq v.
\end{equation}
\end{proof}

\begin{lemma}\label{fromptononp}
If we start with a position $(x,k) \in \mathcal{P}_{k}^f$,
\begin{equation}
move(x,k) \cap \mathcal{P}_{k+1}^f = \emptyset. \label{moveemplty2g}
\end{equation}
\end{lemma}
\begin{proof}
Suppose that 
\begin{equation}
(x,k) \in \mathcal{P}_{k,n}^f \label{xkbelongpmkg}   
\end{equation}
for $n \in \mathbb{Z}_{\ge 0}$  and 
\begin{equation}
(u,k+1) \in move(x,k). \label{uk1belongmoveg}   
\end{equation}
Then, by definitions  \ref{gameofmizugakig} and \ref{defofpkng}, we have
\begin{align}
\sum_{t=1}^{n-1} (f(k+1+2t-1)+1) &  < \sum_{t=1}^n (f(k+2t-2)+1)-f(k) \nonumber \\
& \leq x-f(k) \nonumber \\
& \leq u \nonumber \\
& \leq x-1 \nonumber \\
& \leq \sum_{t=1}^n(f(k+2t-1)+1)-1 \nonumber \\
& < \sum_{t=1}^n(f(k+1+2t-2)+1).\label{uarea}   
\end{align}	
Since 
\begin{equation}
\mathcal{P}_{k+1,n}^f = \{ (x,k):\sum_{t=1}^n (f(k+1+2t-2)+1) \leq x \leq  \sum_{t=1}^n (f(k+1+2t-1)+1) \} \nonumber
\end{equation}
and 
\begin{equation}
\mathcal{P}_{k+1,n-1}^f = \{(x,k):\sum_{t=1}^{n-1} (f(k+1+2t-2)+1) \leq x \leq  \sum_{t=1}^{n-1} (f(k+1+2t-1)+1) \},  \nonumber
\end{equation}
by (\ref{uarea}) we have
\begin{equation}
v < u < w    \label{vuwin}
\end{equation}
for any $(v,k+1) \in \mathcal{P}_{k+1,n-1}^f$ and $(w,k+1) \in \mathcal{P}_{k+1,n}^f$.
By Lemma \ref{comparenn1} and (\ref{vuwin}), we have
$(u,k+1) \notin  \mathcal{P}_{k+1,m}^f = \emptyset$ for any $m \in \mathbb{Z}_{\ge 0}$.
\end{proof}

\begin{lemma}\label{fromntopp}
Suppose that we start with a position 
\begin{equation}
(x,k) \notin \mathcal{P}_{k}^f . \label{xknotinpg}   
\end{equation}
Then,
\begin{equation}
move(x,k) \cap \mathcal{P}_{k+1}^f \ne \emptyset. \label{movenotemg}
\end{equation}
\end{lemma}
\begin{proof}
Suppose that $(x,k) \notin \mathcal{P}_{k}^f$. Then,
there exists $n \in \mathbb{N}$ such that 
\begin{equation}
u < x < v \label{inbetween}
\end{equation}
for any
$(u,k) \in \mathcal{P}_{k,n-1}^f$ and $(v,k) \in \mathcal{P}_{k,n}^f$.
Then, we have
\begin{equation}
\sum_{t=1}^{n-1} (f(k+2t-1)+1) < x < \sum_{t=1}^{n} (f(k+2t-2)+1) \label{inbetween2}
\end{equation}
We prove that we can move to a position in $P^f_{k+1,n-1}$ from $(x,k)$.
By (\ref{inbetween2}), 
\begin{equation}
 x -f(k)  \leq \sum_{t=1}^{n-1} (f(k+1+2t-1)+1) \label{lessthan1}
\end{equation}
and 
\begin{equation}
\sum_{t=1}^{n-1} (f(k+1+2t-2)+1) \leq x-1. \label{morethan1}
\end{equation}
By (\ref{lessthan1}) and (\ref{morethan1}),
\begin{equation}
move(x,k) \cap \mathcal{P}_{k+1,n-1}^f = \{(x-f(k),k+1), \cdots, (x-1,k+1)\} \cap \mathcal{P}_{k+1,n-1}^f \ne \emptyset.    
\end{equation}
\end{proof}

\begin{equation}
\mathcal{P}^f= \cup \{\mathcal{P}_k^f:k \in \mathbb{N}\}.\nonumber
\end{equation}

\begin{theorem}\label{theoremformizugaki}
$\mathcal{P}^f$ is the set of $\mathcal{P}$-\textit{position}s
\end{theorem}
\begin{proof}
This is direct from Lemma \ref{fromntopp} and Lemma \ref{fromptononp}.
\end{proof}

\section{Restricted Nim with Two Kinds of Stones}
In the traditional game of Nim, we use only one kind of stones. Here, we have two kinds of stones with different weights.

\begin{definition}\label{gameoftakahashi}
Suppose there is a pile of stones, and two players take turns removing stones from the pile. There are two types of stones. We call a stone Type 1 when its weight is $1$, and Type 2 when its weight is $2$.
When the total weight of stones is $m \in \mathbb{N}$,
a player is allowed to remove stones whose total weight is less than or equal to 
$\lfloor \frac{m}{2}\rfloor$.
\end{definition}

\begin{definition}\label{takahashi2}
We denote a position of the game of Definition \ref{gameoftakahashi} by 
$(x,y)$, where $x$ and $y$ are numbers of Type 2 stones and Type 2 stones respectively.
\end{definition}

\begin{definition}\label{moveoftakahashi}
For $x,y \in \mathbb{Z}_{\ge 0}$, the set of all the positions that can be reached from position $(x,y)$ is defined as 
$\textit{move}(x,y)$.
Then, we have 
\begin{equation}
\textit{move}(x,y)	= \{(x-t,y-u): t,u \in \mathbb{Z}_{\ge 0} \text{ and } 1 \leq 2t+u \leq \lfloor \frac{2x+y}{2} \rfloor \}.\nonumber   
\end{equation}
\end{definition}

\begin{definition}\label{defoftakahaship}
For $n \in \mathbb{N}$, let
\begin{equation}
\mathcal{P}_{n,1}= \{(2^n-1,0)\},\nonumber
\end{equation}
\begin{equation}
\mathcal{P}_{n,2} = \{ (2^{n}-i-1,2i-1): 1 \leq i \leq n-1 \text{ and }  i \in \mathbb{N}\},
\end{equation}
\begin{equation}
\mathcal{P}_{n,3} = \{ (2^n-n-i,2n +2i - 1):1 \leq i \leq 2^n-n \text{ and } i \in \mathbb{N} \}, \nonumber
\end{equation}
and 
\begin{equation}
\mathcal{P}= \bigcup_{n=1}^{\infty} (\mathcal{P}_{n,1} \cup \mathcal{P}_{n,2} \cup \mathcal{P}_{n,3}).\nonumber
\end{equation}
\end{definition}

\begin{lemma}\label{ptononposition}
Suppose that we start with a position 
\begin{equation}
(x,y) \in \mathcal{P}. \nonumber
\end{equation}
Then,
\begin{equation}
move(x,y) \cap \mathcal{P} = \emptyset. \nonumber
\end{equation}
\end{lemma}
\begin{proof}
Let $m \in \mathbb{N}$. For
positions $(x,y)=(2^m-1,0) \in \mathcal{P}_{m,1}$, $(x,y) =(2^m-i-1,2i-1) \in \mathcal{P}_{m,2}$, and $(x,y)=(2^m-m-i,2m+2i-1) \in \mathcal{P}_{m,3}$,
the total weights of the stones are
\begin{equation}
2(2^m-1)=2^{m+1}-2, \label{totalp1}  
\end{equation}

\begin{equation}
2(2^m-i-1)+2i-1 = 2^{m+1}-3,  \label{totalp2}    
\end{equation}
and 
\begin{equation}
2(2^m-m-i)+2m+2i-1=2^{m+1}-1   \label{totalp3}    
\end{equation}
respectively.\\
$(i)$ Suppose that we start with a position $(x,y)=(2^n-1,0) \in \mathcal{P}_{n,1}$. 
Then, by (\ref{totalp1}), we can remove stones whose total weight is equal to or less than
\begin{equation}
\lfloor \frac{2^{n+1}-2}{2} \rfloor =2^n-1.\label{howmuchtoremove}
\end{equation}
Since the second coordinate $y$ is $0$, it is clear that 
\begin{equation}
 move(x,y) \cap (\mathcal{P}_{n-1,2} \cup \mathcal{P}_{n-1,3}) = \emptyset. \nonumber
\end{equation}
If we are to move to $(2^{n-1}-1,0)$ from $(2^{n}-1,0)$, we need to remove 
stones whose total weight is $2 \times 2^{n-1} = 2^n$.
By (\ref{howmuchtoremove}), this is impossible.\\
Therefore, we have
\begin{equation}
move(x,y) \cap (\mathcal{P}_{n-1,2} \cup \mathcal{P}_{n-1,2} \cup \mathcal{P}_{n-1,3}) = \emptyset. \nonumber
\end{equation}
For any $m \in \mathbb{N}$ such that $m < n-1$, 
it is clear that we cannot move to  positions in $ (\mathcal{P}_{m,2} \cup \mathcal{P}_{m,2} \cup \mathcal{P}_{m,3})$, because 
the total weight of stones of these positions are too small. In the following $(ii)$ and $(iii)$, we omit the part of the proof for $m \in \mathbb{N}$ such that $m < n-1$.\\
$(ii)$ Suppose that we start with a position $(x,y) = (2^n-i-1,2i-1) \in \mathcal{P}_{n,2}$. 
Then, we can remove stones whose total weight is equal to or less than
\begin{equation}
\lfloor \frac{2^{n+1}-3}{2} \rfloor =2^n-2.\label{howmuchtoremove2}
\end{equation}
Then, we can move to a position whose total weight of stones is
equal to or greater than
\begin{equation}
2^{n+1}-3-(2^n-2)=2^n-1.\label{howmuchtoremove2b}
\end{equation}
Then, we have 
\begin{equation}
 move(x,y) \cap (\mathcal{P}_{n-1,1} \cup \mathcal{P}_{n-1,2}) = \emptyset. \nonumber
\end{equation}
Since the second coordinate $y$ is $2i-1$ with $i \leq n-1$,
$y \leq 2n-3$.
Since the second coordinate of a position $(2^{n-1}-(n-1)-i,2(n-1)+2i-1) \in \mathcal{P}_{n-1,3}$ is 
$2(n-1)+2i-1 = 2n+2i-3 > 2n-3$, we have 
\begin{equation}
 move(x,y) \cap \mathcal{P}_{n-1,3} = \emptyset. \nonumber
\end{equation}
$(iii)$ Suppose that we start with a position $(x,y) \in \mathcal{P}_{n,3}$. 
Then, we can remove stones whose total weight is equal to or less than
\begin{equation}
\lfloor \frac{2^{n+1}-1}{2} \rfloor =2^n-1.\label{howmuchtoremove3}
\end{equation}
Then, we can move to a position whose total weight of stones is
equal to or greater than
\begin{equation}
2^{n+1}-1-(2^n-1)=2^n.\label{howmuchtoremove2c}
\end{equation}
Therefore, by (\ref{totalp1}), (\ref{totalp2}), and (\ref{totalp3}),
\begin{equation}
 move(x,y) \cap (\mathcal{P}_{n-1,1} \cup \mathcal{P}_{n-1,2}\cup \mathcal{P}_{n-1,3}) = \emptyset. \nonumber
\end{equation}
\end{proof}
\pagebreak
\begin{figure}[H]
\centering
\includegraphics[height=10cm]{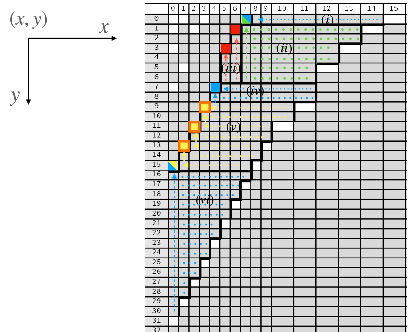}
\caption{How to move to P-positions}\label{positions0}
\end{figure}

\begin{example}
 Figure \ref{positions0} shows how to move to $\mathcal{P}$-positions.\\
$(i)$ If you start with a $\mathcal{N}$-position in Area $(i)$ of Figure  \ref{positions0}, you move rightward to reach $(7,0)$.\\
$(ii)$ If you start with a $\mathcal{N}$-position in Area $(ii)$, you move rightward and go upward to reach $(7,0)$.\\
$(iii)$ If you start with a $\mathcal{N}$-position in Area $(iii)$, you  go upward to reach $(5,3)$ or $(6,1)$.\\
$(iv)$ If you start with a $\mathcal{N}$-position in Area $(iv)$, you  you move rightward and go upward to reach $(4,7)$.\\ 
$(v)$ If you start with a $\mathcal{N}$-position in Area $(v)$, you  you move rightward and go upward to reach $(1,13),(2,11),(3,9)$.\\ 
$(vi)$ If you start with a $\mathcal{N}$-position in Area $(iv)$, you  you move rightward and go upward to reach $(0,15)$.\
\end{example}
Figure \ref{positions0} shows how to move to $\mathcal{P}$-positions, and for each area 
we use different colours.

\begin{figure}[H]
\centering
\includegraphics[height=10cm]{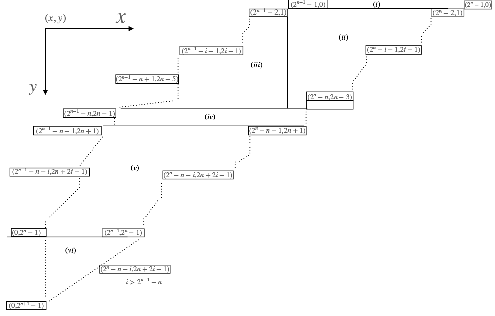}
\caption{How to move to P-positions}\label{positions1}
\end{figure}

\begin{lemma}\label{ntopposition}
Suppose that we start with a position 
\begin{equation}
(x,y) \notin \mathcal{P}. \nonumber
\end{equation}
Then,
\begin{equation}
move(x,y) \cap \mathcal{P} \ne \emptyset. \nonumber
\end{equation}
\end{lemma}
\begin{proof}
Suppose that 
\begin{equation}
 (x,y) \notin \mathcal{P}.   
\end{equation}
$(i)$
Suppose that we start with the position  
$(x,y)=(x,0) \notin \mathcal{P}$, where $2^{n-1}-1 < x < 2^n-1$. 
$(x,0)$ is in Area $(i)$ of Figure \ref{positions1}.
We prove that we can go to the position $(2^{n-1}-1,0) \in \mathcal{P}_{n-1,1}$.
Then, there exists $t  \in \mathbb{N}$ such that
$x=2^n-t-1$ and $1 \leq t \leq 2^{n-1}-1$.
At the position $(x,0)$, the total weight of stones is 
$2x=2^{n+1}-2t-2$, and 
the total weight of stones that we can remove is
\begin{equation}
\lfloor \frac{2^{n+1}-2t-2}{2} \rfloor = 2^n-t-1.\nonumber
\end{equation}
Then, the number of stones of Type 2 that we can remove is
\begin{equation}
\lfloor \frac{2^{n}-t-1}{2} \rfloor  =   \lfloor 2^{n-1} - \frac{t+1}{2} \rfloor. \label{iquation11}
\end{equation}
To move to the position $(2^{n-1}-1,0) \in \mathcal{P}_{n-1,1}$, the number of stones of Type 2 that we need to remove is
\begin{equation}
x-(2^{n-1}-1) = 2^n-t-1-(2^{n-1}-1) = 2^{n-1}-t. \label{iquation22}
\end{equation}
Since $\frac{t+1}{2} \leq t$, by (\ref{iquation11}) and (\ref{iquation22}), we can move to
$(2^{n-1}-1,0) \in \mathcal{P}_{n-1,1}$. \\
$(ii)$
Suppose that we start with the position  
$(x,2i-1) \notin \mathcal{P}$ for $i=1,2,\cdots, n-1$ or $(x,2i-2) \notin \mathcal{P}$ for $i=2,\cdots, n$, where $2^{n-1}-1 \leq x < 2^n-1-i$. The positions $(x,2i-1)$ and $(x,2i-2)$ are in Area $(ii)$ of Figure \ref{positions1}.
Then, there exists $t  \in \mathbb{N}$ such that
$x=2^n-i-t-1$ and $1 \leq t \leq 2^{n-1}-i$.
We prove we can move to the position $(2^{n-1}-1,0) \in \mathcal{P}_{n-1,1}$.
At the position $(x,2i-1)$ or $(x,2i-2)$, the total weight of stones is 
$2x+2i-1 =2(2^n-i-t-1)+2i-1=2^{n+1}-2t-3$ or $2x+2i-2 =2(2^n-i-t-1)+2i-2=2^{n+1}-2t-4$, and 
you can remove stones whose total weight is 
\begin{equation}
\lfloor \frac{2^{n+1}-2t-3}{2} \rfloor  =   2^n-t-2 \label{iquation33}
\end{equation}
or 
\begin{equation}
\lfloor \frac{2^{n+1}-2t-4}{2} \rfloor  =   2^n-t-2 \label{iquation77}
\end{equation}
respectively.
To move to the position $(2^{n-1}-1,0) \in \mathcal{P}_{n-1,1}$ from $(x,2i-1)$ or $(x,2i-2)$, we need to remove
\begin{align}
2(x-(2^{n-1}-1))+2i-1  & = 2(2^n-i-t-1 -(2^{n-1}-1))+2i-1 \nonumber \\
& = 2^n-2t-1 \label{iquation001}
\end{align}	
or 
\begin{align}
2(x-(2^{n-1}-1))+2i-2  & = 2(2^n-i-t-1-2^{n-1}+1)+2i-2 \nonumber \\
& = 2^n-2t-2         \label{iquation88}
\end{align}	
respectively.
Since $t+2 \leq 2t+1 < 2t+2$, by (\ref{iquation33}), (\ref{iquation001}), (\ref{iquation77}) and (\ref{iquation88}),  
we can move to the position $(2^{n-1}-1,0) \in \mathcal{P}_{n-1,1}$.\\
$(iii)$ 
Suppose that we start with the position  
$(x,2i-1) \notin \mathcal{P}$ for $i = 2, \cdots, n-1$ or $(x,2i-2) \notin \mathcal{P}$ for $i = 2, \cdots, n$, where  $2^{n-1}-i-1 < x < 2^n-1$. The positions $(x,2i-1)$ and $(x,2i-2)$ are in Area $(iii)$ of Figure \ref{positions1}.
Then, there exists $t  \in \mathbb{N}$ such that
$x=2^n-t-1$ and $1 \leq t \leq 2^{n-1}+i-1$.
We prove that we can move to the position $(2^n-t-1,2t-1) \in \mathcal{P}_{n-1,1}$.
The total weight of the stones at the position $(x,2i-1)$ is
\begin{align}
2(2^n-t-1)+2i-1  & = 2^{n+1}-2t-2+2i-1 \nonumber \\
& \geq  2^{n+1}+2(i-t)-3.  \nonumber \\
& \geq 2(i-t)+1.   \label{iquation888}
\end{align}	
Similarly, 
The total weight of the stones at the position $(x,2i-2)$ is
\begin{equation}
2(i-t).\label{iquation999}
\end{equation}
If we go to the position $(x,2t-1) = (2^n-t-1,2t-1) \in \mathcal{P}_{n-1,1}$ from 
$(x,2i-1)$ or $(x,2i-2)$, we have to remove stones whose total weight is 
$2(i-t)$ or $2(i-t)-1$ respectively.
Then, by (\ref{iquation888}) and (\ref{iquation999}),
We can got to the position $(2^n-t-1,2t-1) \in \mathcal{P}_{n-1,1}$.\\
$(iv)$ Suppose that we start with the position $(x,2n-1) \notin \mathcal{P}$ or $(x,2n) \notin \mathcal{P}$
, where $2^{n-1}-n < x < 2^n-n$. The positions $(x,2n-1)$ and $(x,2n)$ are in Area $(iv)$ of Figure \ref{positions1}.
We prove that we can move to the position $(2^{n-1}-n,2n-1) \in \mathcal{P}_{n-1,3}$.
Then, there exists $t  \in \mathbb{N}$ such that
$x=2^n-n-t$ and $1 \leq t \leq 2^{n-1}-1$.
At the position $(x,2n-1)$ or $(x,2n)$, the total weight of stones is 
$2x+2n-1 =2(2^n-n-t)+2n-1=2^{n+1}-2t-1$ or $2x+2n =2(2^n-n-t)+2n=2^{n+1}-2t$ 
, and 
you can remove stones whose total weight is 
\begin{equation}
\lfloor \frac{2^{n+1}-2t-1}{2} \rfloor  =   2^n-t-1 \label{iquation113}
\end{equation}
or
\begin{equation}
\lfloor \frac{2^{n+1}-2t}{2} \rfloor  =   2^n-t \label{iquation115}
\end{equation}
respectively.
If we move to the position $(2^{n-1}-n,2n-1)$ from $(x,2n-1)$ or $(x,2n)$,
the total weight of the stones to remove is
\begin{align}
2(x-(2^{n-1}-n))  & = 2(2^n-n-t-2^{n-1}+n) \nonumber \\
& = 2^n-2t      \label{iquation114}
\end{align}	
or
\begin{align}
2(x-(2^{n-1}-n))+1  & = 2(2^n-n-t-2^{n-1}+n)+1 \nonumber \\
& = 2^n-2t+1         \label{iquation116}
\end{align}	
respectively.
By (\ref{iquation113}), (\ref{iquation114}), (\ref{iquation115}) and (\ref{iquation116})
we can move to the position $(2^{n-1}-n,2n-1) \in \mathcal{P}_{n-1,3}$.\\
$(v)$ We start with $(x,2n+2i-1)$ for $i=1,2, \cdots, 2^{n-1}-n$ or $(x,2n+2i)$ for $i=1,2, \cdots, 2^{n-1}-n-1$, 
where $x \in \mathbb{N}$ and $2^{n-1}-n-i < x < 2^n-n-i$. The positions $(x,2n+2i-1)$ and $(x,2n+2i-2)$ are in Area $(v)$ of Figure \ref{positions1}.
We prove that we can go to $(2^{n-1}-(n-1)-(i+1),2(n-1)+2(i+1)-1) 
= (2^{n-1}-n-i,2n+2i-1) \in \mathcal{P}_{n-1,3}$ from
$(x,2n+2i-1)$ and $(x,2n+2i)$ by removing stones.
Note that $(2^{n-1}-n-i,2n+2i-1)$ is the $i+1$th element of $\mathcal{P}_{n-1,3}$.
There exists $t$ such that $1 \leq t \leq 2^{n-1}-1$
and $x = 2^n-n-i-t$.
At the position $(x,2n+2i-1)$ or $(x,2n+2i)$, the total weights of the stones are 
\begin{align}
2x+2n+2i-1  & =2(2^n-n-i-t)+2n+2i-1 \nonumber \\
& =2^{n+1}-2t-1 \nonumber         
\end{align}	
or 
\begin{align}
2x+2n+2i  & = 2(2^n-n-i-t)+2n+2i \nonumber \\
& =   2^{n+1}-2t \nonumber 
\end{align}	
respectively.
Therefore, we can remove stones whose total weight is
\begin{equation}
\lfloor \frac{2^{n+1}-2t-1}{2} \rfloor = 2^n-t-1 \label{totalpp}
\end{equation}
or
\begin{equation}
\lfloor \frac{2^{n+1}-2t}{2} \rfloor = 2^n-t \label{totalpp2}   
\end{equation}
respectively.
If we move to $(2^{n-1}-n-i,2n+2i-1) \in \mathcal{P}_{n-1,3}$ from $(x,2n+2i-1)$ or $(x,2n+2i-2)$,
the total weight of stones to be removed is 
\begin{align}
2(x-(2^{n-1}-n-i))  & = 2(2^n-n-i-t-(2^{n-1}-n-i))    \nonumber \\
& =   2^n-2t     \label{totalpp3}
\end{align}	
or
\begin{align}
2(x-(2^{n-1}-n-i))+1  & = 2(2^n-n-i-t-(2^{n-1}-n-i))+1     \nonumber \\
& =   2^n-2t+1   \label{totalpp4}
\end{align}	
respectively.

by (\ref{totalpp}), (\ref{totalpp2}),  (\ref{totalpp2}), and  (\ref{totalpp4}),
 we can go to
$(2^{n-1}-n-i,2n+2i-1) \in \mathcal{P}_{n-1,3}$.\\
$(vi)$ We start with $(x,2n+2i-1)$ for $i=2^{n-1}-n+1, \cdots, 2^{n}-n-1$, or $(x,2n+2i-2)$ for $i=2^{n-1}-n+1, \cdots, 2^{n}-n$, 
where  $x \in \mathbb{N}$ and $0 < x < 2^n-n-i$. The positions $(x,2n+2i-1)$ and $(x,2n+2i-2)$ are in Area $(vi)$ of Figure \ref{positions1}.
We prove that we can go to $(0,2^n-1) \in \mathcal{P}_{n-1,3}$ from
$(x,2n+2i-1)$ and $(x,2n+2i-2)$ by removing stones.

At the position $(x,2n+2i-1)$ or $(x,2n+2i-2)$, the total weights of the stones are 
\begin{equation}
2x+2n+2i-1 \nonumber
\end{equation}
or
\begin{equation}
2x+2n+2i-2 \nonumber
\end{equation}
respectively.
Therefore, we can remove stones whose total weight is
\begin{equation}
x+n+i-1. \label{totalppb}
\end{equation}
If we move to $(0,2^n-1) \in \mathcal{P}_{n-1,3}$ from $(x,2n+2i-1)$ or $(x,2n+2i-2)$,
the total weight of stones to be removed is 
\begin{equation}
2x+2n+2i-1 -(2^m-1)=2x+2n+2i-2^n \label{totalppb2}
\end{equation}
or
\begin{equation}
2x+2n+2i-2  -(2^m-1)=2x+2n+2i-2^n-1 \label{totalppb3}
\end{equation}
respectively.
Since $2^n-n-i \geq x+1$,
by (\ref{totalppb}), (\ref{totalppb2}), and (\ref{totalppb3})
we can move to $(0,2^n-1) \in \mathcal{P}_{n-1,3}$.
\end{proof}

\begin{theorem}
$\mathcal{P}$ is the set of $\mathcal{P}$-\textit{position}s.
\end{theorem}
\begin{proof}
This is direct from Lemma \ref{ptononposition} and Lemma \ref{ntopposition}   
\end{proof}

\section*{Acknowledgement} 
JSPS KAKENHI Grant Number 23H05173 supported this work.


\end{document}